\documentclass[11pt, reqno]{amsart}

\usepackage[text={6.5in,9.5in},centering]{geometry}
\usepackage{amsmath,amsfonts,amsthm,amssymb,color}
\usepackage[
	colorlinks,
	pdfpagelabels,
	pdfstartview = FitH,
	bookmarksopen = true,
	bookmarksnumbered = true,
	linkcolor = blue,
	plainpages = false,
	hypertexnames = false,
	citecolor = red] {hyperref}
\usepackage{enumitem}
\usepackage{mathrsfs}

\usepackage[T1]{fontenc}

\usepackage{graphicx}
\usepackage[capitalize]{cleveref}
\usepackage{subfigure} 

\makeatletter
     \def\section{\@startsection{section}{1}%
     \z@{.7\linespacing\@plus\linespacing}{.5\linespacing}%
     {\bfseries
     \centering
     }}
     \def\@secnumfont{\bfseries}
     \makeatother

\usepackage{graphicx}

\crefname{enumi}{}{}
\crefname{equation}{}{}

\newcommand{\R}{\mathbb R}
\newcommand{\RR}{\mathbb R}

\newcommand{\E}{\mathbb E}

\newcommand{\HH}{\mathfrak H}

\newtheorem{theorem}{Theorem}[section]
\newtheorem{lemma}[theorem]{Lemma}

\newtheorem{proposition}[theorem]{Proposition}

\theoremstyle{definition}
\newtheorem{definition}[theorem]{Definition}
\theoremstyle{remark}
\newtheorem{remark}{Remark}

\numberwithin{equation}{section}
\begin{document}
\title[SPDEs]{Remarks on existence and uniqueness of the solution for stochastic partial differential equations}

\author[B. Avelin]{Benny Avelin}
\address{Uppsala University, Department of Mathematics, Sweden}
\email{benny.avelin@math.uu.se}

\author[L. Viitasaari]{Lauri Viitasaari}
\address{Aalto University School of Business, Department of Information and Service Management, Finland}
\email{lauri.viitasaari@iki.fi}

\begin{abstract}
In this article we consider existence and uniqueness of the solutions to a large class of stochastic partial differential of form
$\partial_t u = L_x u + b(t,u)+\sigma(t,u)\dot{W}$,
driven by a Gaussian noise $\dot{W}$, white in time and spatial correlations given by a generic covariance $\gamma$. We provide natural conditions under which classical Picard iteration procedure provides a unique solution. We illustrate the applicability of our general result by providing several interesting particular choices for the operator $L_x$ under which our existence and uniqueness result hold. In particular, we show that Dalang condition given in \cite{Da} is sufficient in the case of many parabolic and hypoelliptic operators $L_x$. 
\end{abstract}

\maketitle



\allowdisplaybreaks

\section{Introduction}
In this article we consider the stochastic partial differential equation (SPDE) of form
\begin{equation}
\label{eq:spde-intro}
    \frac{\partial u}{\partial t} (t,x)= (L_x u)(t,x) + b(t,u(t,x))+ \sigma (t,u(t,x)) \dot{W} (t,x), \hskip0.5cm t\geq 0, x \in \mathbb{R}^d
\end{equation}
with initial condition $u(0,x) = u_0(x)$. Here $b$ and $\sigma$ are assumed to be Lipschitz continuous functions and bounded on compacts in the spatial variable $x$, and uniformly in $t$. The stochastic force $\dot{W}$ is assumed to be Gaussian with correlation structure that is white in the time variable $t$ and given by a generic covariance $\gamma$ in the spatial variable $x$.

SPDEs have been a subject of active research in the literature for recent years, and the basic theory is already rather well-established. Especially, initiated by the seminal paper by Dalang \cite{Da}, stochastic heat equations, in which case $L_x = \Delta$ is the Laplace operator, with Gaussian noise that is white in time have received a lot of attention. In this case one obtains the existence and uniqueness of the solution provided that the following so-called Dalang's condition
\begin{equation}
    \label{eq:dalang-heat-intro}
        \int_{\mathbb{R}^d} \frac{1}{\beta+2|\xi|^2}\hat{\gamma}(d\xi) < \infty
\end{equation}
holds for some $\beta_0>0$ (in which case it holds for all $\beta>\beta_0$). Here $\hat{\gamma}$ denotes the non-negative measure arising as the Fourier transform of the spatial covariance $\gamma$ of the noise $\dot{W}$. One also observes that here the term $|\xi|^2$ corresponds to the Fourier multiplier arising from the Laplace operator $\Delta$. This leads to a natural extension to the case where $L_x$ is the $L^2$-generator of a Levy process. In this case \cref{eq:dalang-heat-intro} is replaced by 
\begin{equation}
    \label{eq:dalang-levy-intro}
        \int_{\mathbb{R}^d} \frac{1}{\beta+2\text{Re}\Psi(\xi)}\hat{\gamma}(d\xi) < \infty,
\end{equation}
where $\Psi(\xi)$ arises from the characteristic exponent of the associated Levy process. Here $\text{Re}\Psi(\xi)$ is again a non-negative function, and if \cref{eq:dalang-levy-intro} holds for some $\beta>0$, then it holds for all $\beta>0$. As a particular interesting example, this covers the case of the stochastic fractional heat equation where $L_x$ is given by the fractional Laplace operator $-(-\Delta)^\alpha$, $\alpha\in(0,1]$. Indeed, then $\Psi(\xi) = |\xi|^{2\alpha}$ and one recovers the classical heat equation and condition \cref{eq:dalang-heat-intro} by plugging in $\alpha = 1$. This generalisation is studied, among others, in \cite{Foo-Kho2013} where the existence and uniqueness result was given under suitable assumptions on the coefficients $b,\sigma$. 

A standard technique to prove existence and uniqueness of the solution to \cref{eq:spde-intro} is based on the so-called fundamental solution $G_t(x)$ (or the Green kernel) associated to the equation $\partial_t u = L_x u$. Then one obtains a candidate for the solution through convolutions with the kernel $G_t(x)$. In particular, in the settings mentioned above one obtains that $G_t$ is the density of the underlying Levy process and forms a semigroup. After that, using the positivity of $G_t$ as well, one obtains the solution through Picard iteration. 

In this article we provide a general existence and uniqueness result for the solution to \cref{eq:spde-intro}. Our condition, given by Equation \cref{eq:dalang} below, is similar to \cref{eq:dalang-levy-intro}. However, in our results we do not require non-negativity of the fundamental solutions $G_t$ as we only consider upper bounds for $|G_t|$. Furthermore, we do not require the associated function, given by $\Psi(\xi)$ in \cref{eq:dalang-levy-intro}, to be non-negative. This fact allows to consider a larger class of operators $L_x$ in \cref{eq:spde-intro}, in the case where we do not have a Gaussian upper bound on the fundamental solution. 
Our main contribution is that, by a careful analysis on the essential requirements in the classical arguments, we get a remarkably strong generalization. As such, we are able to cover a host of equations which was previously not known in the literature, including the important Kolmogorov equations. 
Finally, we stress that similar considerations can be applied to equations of form \cref{eq:spde-intro}, where $\partial_t$ is replaced with more general operator $L_t$, see \cref{sec:concluding}.

The rest of the paper is organised as follows. In \cref{sec:results} we present our assumptions and provide the existence and uniqueness result, \cref{thm:main}. In \cref{subsec:examples}, we illustrate the applicability of our result by providing a detailed discussion on several interesting examples. The proof of \cref{thm:main} is postponed to \cref{sec:proofs} where we also recall some basic facts on stochastic calculus. We end the paper with some concluding remarks.

\section{General existence and uniqueness result}
\label{sec:results}
We consider the stochastic partial differential equation 
\begin{equation}
    \label{eq:spde}
    \frac{\partial u}{\partial t} (t,x)= (L_x u)(t,x) + b(t,u(t,x))+ \sigma (t,u(t,x)) \dot{W} (t,x), \hskip0.5cm t\geq 0, x \in \mathbb{R}^d
\end{equation}
where $L_x$ is a suitable differential operator (acting on the variable $x$), and coefficients $b$ and $\sigma$ are assumed to be spatially Lipschitz continuous functions, uniformly in $s$ over compacts. That is, for any $T>0$ we have for all $t\in[0,T]$ that
$$
    |b(t,x) - b(t,y)| \leq L_{b,T}|x-y|
$$
and
$$
    |\sigma(t,x) - \sigma(t,y)| \leq L_{\sigma,T} |x-y|.
$$
Note that then the functions $b$ and $\sigma$ are spatially locally bounded, uniformly in $s$ over compacts. That is, 
$$
    \sup_{s\in[0,T],y\in K} \max(|\sigma(s,y)|,b(s,y)|) < \infty
$$
for all compact sets $K \subset \mathbb{R}^d$ and all finite $T>0$.
For the (centered) Gaussian noise $\dot{W}$ we assume that the covariance is given by 
$$
    \E \left[\dot{W}(t,x)\dot{W}(s,y)\right] = \delta_0(t-s)\gamma(x-y),
$$
where $\delta_0$ denotes the Dirac delta and $\gamma$ are non-negative and non-negative definite measures. That is, the noise can be described by a centered Gaussian family with covariance
\begin{equation}
    \label{eq:cov-fourier}
    \E [W(\phi)W(\psi)] = \int_0^\infty \int_{\mathbb{R}^d} \mathcal{F}\phi(\cdot,s)(\xi) \overline{\mathcal{F}(\psi)(\cdot,s)}(\xi)\hat{\gamma}(d\xi)ds,
\end{equation}
where $\hat{\gamma}$ is the spectral measure of $\gamma$, $\mathcal{F}$ denotes the Fourier transform, and $\phi,\psi$ are suitable functions. We provide rigorous treatment on the construction of the Gaussian noise $\dot{W}$ in \cref{subsec:stochastic-calculus} that contains a brief introduction to Gaussian analysis and stochastic integration required for our analysis.

We denote by $G_t(x)$ the fundamental solution (or the Green kernel) associated to the operator $M = \partial_t - L_x$, in the sense that for any "nice enough" intial data $u_0$ we have that
\begin{align*}
    u(x,t) = \int G_{t}(x-y) u_0(y) dy
\end{align*}
satisfies $Mu = 0$ and $u \to u_0$ as $t \to 0$.
% We denote by $G_t(x)$ the fundamental solution (or the Green kernel) associated to the equation \cref{eq:spde}, i.e. $G_t(x)$ is the solution to the deterministic equation 
% \begin{equation}
%     \label{eq:pde-det}
%     \frac{\partial u}{\partial t} = (L_x u)(t,x)
% \end{equation}
% with initial condition $u(0,x) = \delta_0(x)$, where $\delta_0$ denotes the Dirac delta. 

Consider now our original equation \cref{eq:spde}. For the initial condition $u(0,x) = u_0(x)$, we assume that $u_0(x)$ is deterministic and satisfies, for every $T>0$,
\begin{equation}
    \label{eq:initial-bound}
    \sup_{t\in (0,T],x\in \mathbb{R}^d} \int_{\mathbb{R}^d} |G_t(x-y)||u_0(y)|dy < \infty.
\end{equation}

We prove that under certain conditions, equation \cref{eq:spde} admits a unique mild solution in the following sense.
\begin{definition}
    We say that a random field $u(t,x)$, adapted to the filtration generated by $\dot{W}$, is a mild solution to \cref{eq:spde} if for all $(t,x) \in \mathbb{R}_+ \times \mathbb{R}^d$, the process $(s,y) \rightarrow G_{t-s}(x-y)\sigma(u(s,y))\textbf{1}_{[0,t]}(s)$ is integrable and we have
    \begin{eqnarray}
        u(t,x) &=&\int_{\mathbb{R}^d} G_{t}(x-y)u_0(y)dy+ \int_{0} ^ {t} \int_{\mathbb{R}^d}G_{t-s}(x-y)b(s,u(s,y))dyds \nonumber.\\
        &+& \int_{0} ^ {t} \int_{\mathbb{R}^d} G_{t-s} (x-y) \sigma (s,u(s, y)) W (ds, dy).\label{eq:mild}
    \end{eqnarray}
\end{definition}
Here the stochastic integral exists in the sense of Dalang-Walsh, see \cref{subsec:stochastic-calculus}.

The following existence and uniqueness result is the main result of this paper. The proof follows standard arguments and is presented in \cref{sec:proofs}. 
\begin{theorem}
    \label{thm:main}
    Let $G_s$ be the fundamental solution to \cref{eq:spde} and assume that there exists an integrable function $g_s(x) = \mathcal{F}^{-1}(e^{-s \hat f})$, $\hat f \geq -C$ such that $|G_s(x)| \leq g_s(x)$ and there exists $\beta_0>2C$ such that, 
    \begin{equation}
        \label{eq:dalang-deterministic}
        \int_0^\infty e^{-\beta_0 s}\Vert g_s\Vert_{L^1(\mathbb{R}^d)}ds < \infty
    \end{equation}
    and
    \begin{equation}
        \label{eq:dalang}
        \int_{\mathbb{R}^d} \frac{1}{\beta_0 + 2\hat f(\xi)} \hat \gamma (d\xi) < \infty
    \end{equation}
    Then \cref{eq:spde} admits a unique mild solution.
\end{theorem}
\begin{remark}
    \label{remark:b-condition}
    Observe that the condition \cref{eq:dalang-deterministic} is independent of the chosen covariance $\gamma$, and is rather mild. Indeed, in many examples $g_s$ is given by a density of some generating stochastic process at time $s$, and hence we have $\Vert g_s\Vert_{L^1(\mathbb{R}^d)} = 1$ for all $s>0$. Moreover, by carefully examining our proof one observes that \cref{eq:dalang-deterministic} can be replaced with a weaker condition $\int_0^T \Vert g_s\Vert_{L^1(\mathbb{R}^d)}ds < \infty$ for all $T>0$ finite. Finally, we note that \cref{eq:dalang-deterministic} can be omitted in the case $b\equiv 0$. 
\end{remark}
\begin{remark}
    The proof of \cref{thm:main} actually gives more, however for purposes of exposition we chose to present the simplified form \cref{eq:dalang}. For a discussion about further cases (including for instance the wave equation) see, \cref{sec:concluding}.
\end{remark}
\begin{remark}
    Note that if \cref{eq:dalang} is satisfied for some $\beta_0>0$, then it is automatically valid for all $\beta>\beta_0$ as well. 
\end{remark}
\subsection{Examples}
\label{subsec:examples}
In this section we outline some examples of quite general linear equations for which \cref{thm:main} can be applied.
\subsubsection{Linear operators with Gaussian upper bounds}
\label{sssec:Gaussian}
Many parabolic and hypoelliptic evolution operators satisfy a Gaussian upper bound for their fundamental solution, namely an estimate of the following type
\begin{align} \label{eq:Gaussian}
    G_s(x) \leq g(t)e^{-\frac{C |x|_{\mathbb{G}}^2}{t}}
\end{align}
where $\mathbb{G}$ is a homogeneous Lie group. The reason for writing it as above is that if we consider $|x|_{\mathbb{G}}= |x|$, being the standard Euclidean norm, we can cover a big class of parabolic operators. But for hypoelliptic equations, we get a nontrivial norm induced by the corresponding Lie-group.

Let us now consider some interesting examples satisfying \cref{eq:Gaussian} with the Euclidean norm.

Let $L$ in $\R^n$ be given as
\begin{align*}
    L = \sum_{i,j=1}^n a_{ij}(x,t) \partial_{ij} + \sum_{j=1}^n b_i(x,t) \partial_i + c(x,t)
\end{align*}
then if the matrix $a_{ij}(x,t)$ is uniformly elliptic, the functions $a,b,c$ are bounded and uniformly Hölder continuous with exponent $\alpha$, then there exists positive constants $c_1,c_2$ such that
\begin{align}
    \label{eq:L:gaussian}
    |G_t(x)| \leq c_1 t^{-n/2} e^{-c_2 \frac{|x|^2}{t}}
\end{align}
where $G_t(x)$ is the fundamental solution to $\partial_t - L$, see \cite{Friedman1,Friedman2}. As such \cref{thm:main} is applicable, and since the Fourier transform of a Gaussian is a Gaussian we get that if
\begin{align} \label{eq:dalang:gaussian}
    \int_{\R^d} \frac{1}{\beta_0 + 2C(c_1)|\xi|^2} \hat \gamma(d \xi) < \infty
\end{align}
 then there exists a mild solution to \cref{eq:spde}. That is, the condition reduces to the standard one for the heat equation, \cite{Da}.
For instance, if $\gamma$ is given by the Riesz kernel, i.e. $\hat \gamma = |\xi|^{\lambda-d} d\xi$, then the above is verified as long as $\lambda < 2$.
In the case when $\gamma$ is in $W^{k,1}(\R^n)$ then
\begin{align*}
    |\hat \gamma| \leq \frac{C}{(1+|\xi|)^k},
\end{align*}
from this we see that \cref{eq:dalang:gaussian} is verified if $k > n-2$. Finally we remark that if $\gamma = \delta_0$ (white noise), then \cref{eq:dalang:gaussian} only holds in dimension 1.

In fact the above can be extended as follows, if we instead consider the divergence form operator
\begin{align*}
    L = \sum_{i,j=1}^n \partial_i (a_{ij}(x,t) \partial_{j} \cdot ),
\end{align*}
with $a_{ij}$ still an elliptic matrix but now only bounded and measurable, then $\partial_t-L$, Aronson \cite{Aronson} proved that \cref{eq:L:gaussian} holds also in this case.

The case for non-divergence form equations with rough coefficients was treated in \cite{Escauriaza}.

Let us now consider the hypoelliptic setting and consider some relevant examples. Let us begin with some notation. We consider the hypoelliptic evolution operator
\begin{align*}
    M = \sum_{i = 1}^m X_i^2 + X_0 - \partial_t,
\end{align*}
where $X_i$ are smooth vector fields in $\R^n$ for $i=0,\ldots,m$, and usually $m < n$. Such an equation induces an interesting geometry. Namely, let us denote
\begin{align*}
    Y = X_0 - \partial_t, \quad \text{and} \quad \lambda \cdot X = \lambda_1 X_1 + \ldots \lambda_m X_m
\end{align*}
then a curve in $\gamma: [0,R] \to \R^n \times [0,T]$ is $M$-admissible if it is absolutely continuous and
\begin{align*}
    \gamma'(s) = \lambda(s) \cdot X(\gamma(s)) + Y(\gamma(s)), \quad \text{a.e. in [0,R]}.
\end{align*}
That is, the curve has a direction at each point in the tangent-bundle described by the vector-fields $X,Y$. If we assume that the vector fields are such that 
\begin{itemize}
    \item every two points $(x,t), (\xi,\tau) \in \R^n \times [0,T]$ can be connected with an $M$ admissible curve,
    \item there exists a homogeneous Lie group $\mathbb{G} = (\R^{n} \times [0,T], \circ, \delta_\lambda)$ for which $X,Y$ are left translation invariant. Furthermore, that $X$ is $\delta_\lambda$-homogeneous of degree 1 and $Y$ is $\delta_\lambda$-homogeneous of degree 2.
\end{itemize}
Then the operator $M$ is hypoelliptic (distributional solutions are smooth), see \cite{Lanconelli,Kogoj}.
This allows us to write $\R^n = V_1 \oplus \ldots \oplus V_l$, i.e. the vector field $X$ stratifies $\R^n$ into a direct sum of subspaces, such that if we write $x = x^{(1)} + \ldots + x^{(l)}$, where $x^{(k)} \in V_k$, then the dilations becomes simple multiplication in the sense that
\begin{align*}
    \delta_\lambda(x,t) = (\lambda x^{(1)} + \ldots + \lambda^l x^{(l)}, \lambda^2 t).
\end{align*}
Furthermore if we introduce the $\delta_\lambda$ homogeneous norm as
\begin{align*}
    |x|_{\mathbb{G}} = \max \left \{ |x_i^{(k)}|^{\frac{1}{k}}, k=1,\ldots, l, i=1,\ldots, m_k \right \}
\end{align*}
Then, for such operators $M$ we have that there exists (see \cite{Kogoj}) a positive constant $C$ such that
\begin{align*}
    |G_t(x)| \leq \frac{C}{t^{\frac{Q-2}{2}}}e^{-\frac{|x|_{\mathbb{G}}^2}{C t}}.
\end{align*}
From this we see that we can apply \cref{thm:main} again in this context, however now $\gamma$ needs to be adapted to the geometry of $\mathbb{G}$. Furthermore, for certain hypoelliptic operators we even have positivity of the fundamental solution, \cite{A1}.

A prototypical example of a hypoelliptic operator is the following (well known Kolmogorov operator)
\begin{align*}
    K = \sum_{j=1}^n \partial_{x_j}^2 + \sum_{j=1}^n x_j \partial_{y_j} - \partial_t
\end{align*}
where we are working in $\R^{2n} \times [0,T]$ and we denote the first $n$ coordinates as $x_i$ and the other $n$ as $y_i$. The operator $K$ is hypoelliptic in the above sense \cite{Lanconelli}. Furthermore, the fundamental solution for $n=1$ is given by
\begin{align*}
    G_t(x,y) = \frac{\sqrt{3}}{2 \pi t^2} \exp \left (-\frac{x^2}{t} - \frac{3xy}{t^2} - \frac{3 y^2}{t^3} \right )
\end{align*}
Considering the above "heat type" kernel in the finite interval $[0,T]$, it is clear that we can bound it from above by $g_t(x) = \frac{C}{t^2} \exp \left (-\frac{1}{C}\frac{x^2+y^2}{t}\right )$ for a positive constant $C$. As such the same conclusion as in \cref{sssec:Gaussian} holds.

Operators of this type often appear in the context of SDE's where the noise is only in some directions, like the kinetic equations, for which the density satisfies the kinetic Fokker-Planck equation. The kinetic Fokker-Planck is hypoelliptic, see \cite{Villani} and the references therein.

\subsubsection{Fractional heat equation}
\label{sssec:fractional}
Consider the fractional evolution equation defined as
\begin{align*}
    \partial_t u = -(-\Delta)^s u + Cu, \quad 0 < s \leq 1,
\end{align*}
where $C = 0$ corresponds to the fractional heat equation. The fourier transform of said equation is
\begin{align*}
    \partial_t \hat u = -|\xi|^{2s} \hat u + C \hat u.
\end{align*}
As such, the fundamental solution is given by $\hat{G}_t(\xi) = e^{-(|\xi|^{2s} - C)t}$, and we can thus apply \cref{thm:main} with $b \equiv 0$ and $\hat f \geq -C$, together with \cref{remark:b-condition}.

% For $C=0$, the fundamental solution is bounded above and below as
% \begin{align*}
%     G_s(x) \eqsim \frac{t^{-\frac{d}{2s}}}{(1+|t^{-\frac{1}{2s}}x|)^{d+2s}}
% \end{align*}
% see \cite{CD,DD}. Note that in this case we can take $g_s(x) := G_s(x)$ and apply \cref{thm:main} to get that for covariance functions $\gamma$ satisfying
% \begin{align*}
%     \int_{\R^d} \frac{1}{\beta + 2 |\xi|^{2s}} \hat \gamma(d \xi) < \infty,
% \end{align*}
% there exists a unique mild solution to \cref{eq:spde}. 

\subsubsection{Mixture operators}
\label{sssec:mixture}
We formally consider an equation of the following type
\begin{align*}
    \partial_t u = L_1 u + L_2 u.
\end{align*}
If $G_1,G_2$ are the fundamental solutions of
\begin{align*}
    \partial_t u = L_1 u, \quad \partial_t u = L_2 u,
\end{align*}
respectively, then $G = G_1 \ast G_2$ is a fundamental solution of the above, (under some integrability assumptions). Indeed, we have
\begin{align*}
    \partial_t (G_1 \ast G_2) &= (\partial_t G_1 \ast G_2) + (G_1 \ast \partial_t G_2)= (L_1 G_1) \ast G_2 + G_1 \ast (L_2 G_2) \\
    &= (L_1 + L_2) (G_1 \ast G_2).
\end{align*}

Thus for instance we can combine \cref{sssec:Gaussian,sssec:fractional} and apply \cref{thm:main} to get the existence of a mild solution to \cref{eq:spde}.

\section{Proof of \cref{thm:main}}
\label{sec:proofs}
\subsection{Preliminaries on stochastic calculus}
\label{subsec:stochastic-calculus}
In this section we introduce stochastic analysis with respect to the noise $\dot{W}$.

Denote by $C_{c}^{\infty}\left( [0, \infty) \times \mathbb{R}^d\right)$ the class of $C^{\infty}$ functions on $ [0, \infty) \times \mathbb{R}^d$ with compact support.  We consider a Gaussian family of centered random variables  
\[
\left( W(\varphi), \varphi \in C^{\infty}_{c} \left( [0, \infty) \times \mathbb{R}^d\right)\right)
\]
on some complete probability space $\left(\Omega, \mathcal{F}, P\right) $ such that 
\begin{eqnarray}
&&\E [W(\varphi) W (\psi) ] \nonumber \\
&=& \int_{0} ^{\infty}  \int_{\mathbb{R}^d} \int_{\mathbb{R}^d}\varphi(s, y) \psi (s, y') \gamma(y-y')dy dy'ds := \langle \varphi, \psi \rangle _{\HH}.\label{cov-col}
\end{eqnarray}
By taking the Fourier transform, this can equivalently be written as \cref{eq:cov-fourier}. 
Note that here $\gamma$ is not a function, and hence \cref{cov-col} should be understood as 
\begin{equation}
\label{cov-col-general}
\E [W(\varphi) W (\psi) ]
= \int_{0}^{\infty} \int_{\mathbb{R}^d} \varphi(s, y) \left[\psi(s,\cdot)\ast \eta\right](y) dy ds,
\end{equation}
where $\ast$ denotes the convolution. 
For the simplicity of our presentation, we use notation $\gamma(y-y')dydy'$ in \cref{subsec:auxiliary-results,subsec:main-proof} from which change of variable transformations are easier to follow. 

We denote by $\HH $ the Hilbert space defined as the closure of $C_{c}^{\infty}\left( [0, \infty) \times \mathbb{R}^d\right)$ with respect to the inner product (\cref{cov-col}). As a result, we obtain an isonormal process $(W(\varphi), \varphi \in \HH )$, which consists of  a Gaussian family of centered random variable such that, for every $\varphi, \psi \in \HH$,
\begin{equation*}
\E[W(\varphi) W (\Psi) ]= \langle \varphi, \psi \rangle _{\HH}.
\end{equation*}
The filtration $\mathbb{F} = (\mathcal{F}_t)_{t\geq 0}$ associated to the random noise $W$ is generated by random variables $W(\varphi)$ for which $\varphi \in \HH$ has support contained in $[0,t] \times \mathbb{R}^d$. 

Let us now define the stochastic integral with respect to $\dot{W}$. For every random field $\{X(s,y),s\geq 0,y\in \mathbb{R}^d\}$ such that 
$$
\E\Vert X\Vert_{\HH}^2  = \E \int_0^\infty \int_{\mathbb{R}^d} X(s,y)X(s,y')\gamma(y-y')dydy' < \infty,
$$
we can define the stochastic integral
$$
\int_0^\infty \int_{\mathbb{R}^d} X(s,y)W(ds,dy)
$$
in the sense of Dalang-Walsh (see, e.g. \cite{Da,Walsh}). It follows that we have the Isometry 
\begin{equation}
\label{eq:isometry}
\E \left[ \int_0^\infty \int_{\mathbb{R}^d} X_n(s,y) W(ds,dy) \right]^2 = \E\Vert X_n\Vert_{\HH}.
\end{equation}
Moreover, we have the following version of the Burkholder-Davis-Gundy inequality: for any $t\geq 0$ and $p\geq 2$,
\begin{eqnarray}
&&\left| \left|  \int_{0} ^{\infty} \int_{\mathbb{R}^d} X(s, y) W (ds, dy)\right|\right|  _{p} ^{2}\nonumber \\
&&\leq c_{p} \int_{0} ^{t} \int_{\mathbb{R}^d}\int_{\mathbb{R}^d }\Vert X(s, y) X(s, y')\Vert  _{\frac{p}{2}} \gamma(y-y')dydy'ds. \label{burk2}
\end{eqnarray}

\subsection{Auxiliary results}
\label{subsec:auxiliary-results}
Set 
\begin{equation}
\label{eq:I}
I(s) = \Vert g_s\Vert_{L^1(\mathbb{R}^d)} + \int_{\mathbb{R}^d} \int_{\mathbb{R}^d} g_s(y)g_s(y')\gamma(y-y')dydy'.
\end{equation}
By taking the Fourier transform, we have
\begin{equation*}
I(s) = \Vert g_s\Vert_{L^1(\mathbb{R}^d)}+\int_{\mathbb{R}^d}|\hat{g}_s(\xi)|^2\hat{\gamma}(d\xi) = \Vert g_s\Vert_{L^1(\mathbb{R}^d)}+ \int_{\mathbb{R}^d}e^{-2s\hat{f}(\xi)} \widehat{\gamma}(d\xi).
\end{equation*}
\begin{lemma}
Suppose that \cref{eq:dalang-deterministic}-\cref{eq:dalang} hold for some $\beta_0$ and set
$$
\Upsilon(\beta) = \int_0^\infty e^{-\beta s}I(s)ds.
$$
Then $\Upsilon: (\beta_0,\infty) \rightarrow (0,\infty)$ is well-defined and non-increasing in $\beta$. Moreover, $\lim_{\beta\to \infty}\Upsilon(\beta) = 0$.
\end{lemma}
\begin{proof}
By Tonelli's theorem, we have
\begin{equation*}
  \int_0^\infty e^{-\beta s}\int_{\mathbb{R}^d}e^{-2s\hat{f}(\xi)} \widehat{\gamma}(d\xi)ds = \int_{\mathbb{R}^d}\int_0^\infty e^{-\beta s-2s\hat{f}(\xi)}ds \hat{\gamma}(d\xi)=\int_{\mathbb{R}^d}\frac{1}{\beta+2\hat{f}(\xi)}\hat{\gamma}(d\xi)
\end{equation*}
for every $\beta$ such that $\beta+2\hat{f}(\xi)>0$. Since $I(s)$ is non-negative, it is clear that $\Upsilon(\beta)$ is non-increasing which concludes the proof.
\end{proof}
\begin{remark}
We remark that for our purposes, it would actually suffice to consider $$
\Upsilon_T(\beta) = \int_0^T e^{-\beta s}I(s)ds
$$
for each fixed $T<\infty$. Hence we could replace the condition \cref{eq:dalang-deterministic} with $\int_0^T \Vert g_s\Vert_{L^1(\mathbb{R}^d)}ds < \infty$, cf. \cref{remark:b-condition}.
\end{remark}
The following Proposition is the main technical ingredient. The result follows directly from \cite[Lemma 15]{Da} adapted to our context.
\begin{proposition}
\label{prop:key}
Let $I(s)$ be given by \cref{eq:I} and suppose that \cref{eq:dalang-deterministic}-\cref{eq:dalang} hold. Let $\iota>0$, $\beta>\beta_0$, and $T\in(0,\infty)$ be fixed, and let $h_n$ be a sequence of non-negative functions such that $\sup_{t\in[0,T]} h_0(t)< \infty $ and, for $n\geq 1$, we have
\begin{equation}
\label{eq:recursion}
h_{n}(t) \leq \iota\int_0^t  h_{n-1}(s)e^{-\beta(t-s)}I(t-s)ds.
\end{equation}
Then the series
\begin{equation}
\label{eq:lp-convergence}
H(\iota,p,t) :=\sum_{n\geq 0} h_n(t)
\end{equation}
converges uniformly in $t\in[0,T]$.
\end{proposition}
\subsection{Proof of \cref{thm:main}}
\label{subsec:main-proof}
\begin{proof}[Proof of \cref{thm:main}] 
Let $T>0$ be fixed and finite. We consider the standard Picard iterations by setting $u_{0}(t,x)=\int_{\mathbb{R}^d} G_t(x-y)u_0(y)dy$ and, for $n\geq 1$ and a given $\beta>\beta_0$,
\begin{equation*}
\begin{split}
e^{-\beta t} u_{n+1} (t,x)&=e^{-\beta t}u_{0}(t,x)+ \int_0^t e^{-\beta t}\int_{\mathbb{R}^d} G_{t-s}(x-y)b(s,u_n(s,y))dyds\\
&+\int_{0} ^ {t}e^{-\beta t}\int_{\mathbb{R}^{d}} G_{t-s}(x -y ) \sigma (s,u_{n}(s, y))W (ds, dy), \hskip0.5cm t\geq 0, x\in \mathbb{R}^{d}.
\end{split}
\end{equation*} 
We first prove that, for each $n\geq 1$, $u_{n}(t,x)$ is well-defined and satisfies 
\begin{equation}\label{eq:sup-exp-norm-n}
\sup_{t\in (0,T]}\sup_{x\in \mathbb{R}^d}  \E\left[e^{-p\beta t} \left| u_{n}(t,x) \right| ^ {p} \right]<\infty.
\end{equation}
We first note that if \cref{eq:sup-exp-norm-n} holds for some $n\geq 0$, it follows 
from the Lipschitz continuity of $\sigma$ that 
\begin{equation*}
    \begin{split}
        \E \left[e^{-p\beta s}\vert  \sigma(s,u_{n}(s,y)) \vert ^p\right] &\leq C \left[\E \left[e^{-p\beta s}\vert  \sigma(s,u_{0}(s,y)) \vert ^p\right]\right. \\
        &+\left. \E \left[e^{-p\beta s}\vert  u_{n}(s,y))-u_0(s,y) \vert ^p\right]\right].
    \end{split}
\end{equation*}
By the boundedness of $\sigma$ on compacts and \cref{eq:initial-bound}, we get
$$
\sup_{s\in[0,T],y\in\mathbb{R}^d} \E \left[e^{-p\beta s}\vert  \sigma(s,u_{0}(s,y)) \vert ^p\right] = \sup_{s\in[0,T],y\in\mathbb{R}^d} e^{-p\beta s}\vert  \sigma(s,u_{0}(s,y)) \vert ^p< \infty
$$
and consequently, 
$$
\E \left[e^{-p\beta s}\vert  \sigma (s,u_{n}(s,y)) \vert ^p\right] \leq    C\left(1+ \sup_{s\in [0,T],y\in \mathbb{R}^d} \E \left[e^{-p\beta s}\vert u_{n}(s, y)\vert ^p  \right]  \right) <\infty.
$$ 
This in turn gives us, by H\"older's inequality, that
$$
\sup_{s\in [0,T]}\sup_{y,y'\in \RR^d} e^{-\beta s}\left| \left| \sigma(s,u_{n}(s, y)) \sigma (s,u_{n}(s, y'))\right| \right| _{\frac{p}{2}} < \infty.
$$
Exactly the same way, we obtain 
$$
\sup_{s\in [0,T]}\sup_{y\in \RR^d}\E \left[e^{-p\beta s}\vert  b (s,u_{n}(s,y)) \vert ^p\right] < \infty.
$$
In view of \cref{eq:isometry} and \cref{burk2} together with the boundedness of $u_0(t,x)$ and $|G_t(y)|\leq g_t(y)$, applying the Minkowski integral inequality leads to  
\begin{eqnarray*}
&&\E \left[e^{-p\beta t}\left|u_{n+1}(t,x)\right|^p \right] \\
&\leq &  C\left( \left[e^{-\beta t}u_0(t,x)\right]^p+   \Big\| \int_{0} ^{t} e^{-\beta t} \int_{\mathbb{R}^d} G_{t-s}(x-y) b(s,u_n(s,y))dy ds \Big\|_p^p\right. \\
&+& \left. \Big\| \int_0^t e^{-\beta t}\int_{\mathbb{R}^{d}} \int_{\mathbb{R}^{d}}G_{t-s}(x -y )G_{t-s} (x -y' ) \right.\\
&&\phantom{kukkuu}\left.\times\sigma(s,u_{n}(s, y)) \sigma (s,u_{n}(s, y')) \gamma(y-y ')dy'dyds\Big\| _{\frac{p}{2}}^{\frac{p}{2}}\right)\\
&\leq &  C\left[ 1+   \left(\int_{0} ^{t} e^{-\beta (t-s)} \int_{\mathbb{R}^d} g_{t-s}(x-y) e^{-\beta s}\Vert b(s,u_n(s,y))\Vert_{p}dyds\right)^{p}  \right. \\
&+& \left. \left( \int_0^t e^{-\beta (t-s)}\int_{\mathbb{R}^{d}} \int_{\mathbb{R}^{d}}g_{t-s} (x -y )g_{t-s} (x -y' )\right.\right.\\
&&\phantom{kukkuu}\left.\left.\times
e^{-\beta s}\left| \left| \sigma(s,u_{n}(s, y)) \sigma (s,u_{n}(s, y'))\right| \right| _{\frac{p}{2}}\gamma(y-y ')dy'dyds\right) ^{\frac{p}{2}}\right]\\
&\leq &  C \left[ 1+ \left(\int_{0}^{t} e^{-\beta(t-s)}I (t-s) ds \right)^{p}+\left(\int_{0}^{t} e^{-\beta(t-s)}I (t-s) ds \right)^{\frac{p}{2}}\right] \\
& \leq & C \left(1+ \Upsilon^p(\beta)+\Upsilon^{\frac{p}{2}}(\beta)\right).
\end{eqnarray*}
which is finite for $\beta>\beta_0$. Since $u_0(t,x)$ is uniformly bounded on $[0,T] \times \mathbb{R}^d$, \cref{eq:sup-exp-norm-n} for all $n\geq 0$ thus follows from induction and hence $u_{n+1}$ is well-defined. The same arguments applied to 
$$
H_{n}(t):= \sup_{x \in \mathbb{R}^{d}} \left[e^{-\beta pt}\E \left[\left| u_{n+1} (t, x)- u_{n}(t,x) \right| ^ {p}\right]\right]^{\frac{1}{p}}
$$
gives us
\begin{align*}
H_n(t)&\leq  C \left[ \int_0^t e^{-\beta(t-s)}\int_{\mathbb{R}^d}g_{t-s}(x-y)H_{n-1}(s)dyds\right.  \\
&+\left( \int_{0} ^ {t} \int_{\mathbb{R}^{d}  }  \int_{\mathbb{R}^{d}  } e^{-2\beta t} g_{t-s} (x -y ) g_{t-s}(x -y' )\left| \left| \sigma (s,u_{n}(s, y))-  \sigma (s,u_{n-1}(s, y))\right|\right| _{p}\right.\\
&\phantom{kukkuu}\times\left.\left.
  \left|  \left| \sigma (s,u_{n}(s, y'))-  \sigma (s,u_{n-1}(s, y'))\right|\right|_{p} \gamma(y-y')dy'dyds\right) ^ {\frac{1}{2}}\right]\\
  &\leq  C \left[ \int_0^t e^{-\beta(t-s)}I(t-s)H_{n-1}(s)ds+\left( \int_{0} ^ {t} e^{-2\beta (t-s)}I(t-s)H^2_{n-1}(s)ds\right) ^ {\frac{1}{2}}\right].
\end{align*}
Here H\"older's inequality gives 
$$
\left( \int_{0} ^ {t} e^{-2\beta (t-s)}I(t-s)H^2_{n-1}(s)ds\right) ^ {\frac{1}{2}} \leq C \int_0^t e^{-\beta(t-s)}I(t-s)H_{n-1}(s)ds
$$
leading to
$$
H_n(t) \leq C \int_0^t e^{-\beta(t-s)} I(t-s)H_{n-1}(s)ds.
$$
By noting that $b(s,u_0(s,y))$ and $\sigma(s,u_0(s,y))$ are uniformly bounded over $[0,T]\times \mathbb{R}^d$, we obtain from the above computations that $\sup_{s\in[0,T]} H_0(s) < \infty$. Consequently, it follows from \cref{prop:key} that $\sum_{n\geq 1} H_{n} (t)$ converges uniformly on $[0, T]$, and hence the sequence $u_{n}$ converges in $L^ {p}(\Omega) $ to some process $u \in L^p(\Omega)$, uniformly on $[0, T] \times \mathbb{R}^{d} $. From the uniform convergence one can deduce further that $u$ satisfies \cref{eq:mild}, and thus we have obtained the existence of the mild solution. Proving the uniqueness in a similar fashion and noting that $T>0$ was arbitrary concludes the proof.
\end{proof}

\section{Concluding remarks}
\label{sec:concluding}
In this article we have considered general SPDEs \cref{eq:spde} and showed that one can obtain existence and uniqueness of the mild solution by classical arguments, initiated by Dalang \cite{Da}, for a very large class of operators $L_x$. Our condition \cref{eq:dalang} is 
similar to the Dalang condition \cref{eq:dalang-heat-intro} for the stochastic heat equation and actually, our examples in \cref{subsec:examples} reveals that \cref{eq:dalang-heat-intro} is indeed sufficient for a very large class of operators $L_x$. By carefully examining the above proofs, we observe that even more is true. Indeed, condition \cref{eq:dalang} could be replaced with a condition $|G_s(x)|\leq g_s(x)$ and, for all $T>0$, 
    \begin{equation*}
        \int_0^T \int_{\mathbb{R}^d} e^{-\beta s}[\hat{g}_s(\xi)]^2 \hat \gamma (d\xi) ds < \infty.
    \end{equation*}
This formulation allows to consider even more general operators $L_t$ in the time variable instead of considering merely $L_t = \partial_t$. In particular, with this formulation one can easily recover the existence and uniqueness result of \cite{Da} for the stochastic wave equation (in one or two dimensions) with $L_t = \partial_{tt}$. As another example, we can obtain the existence and uniqueness result in the case of the fractional power of the full heat operator $(\partial_t - \Delta)^s$, studied for instance in \cite{NS}.

\end{document}